\documentclass{article}

\usepackage[utf8]{inputenc}
\usepackage{amsmath, amssymb, amsthm}
\usepackage{graphicx}
\usepackage{geometry}
\usepackage{tikz}
\usepackage{tikz-cd}
\usepackage{multirow}
\usepackage{hyperref}
\usepackage{cleveref}  % 在hyperref之后加载
\usepackage{forest}   
\usepackage{mathrsfs} 
\usepackage{setspace}
\title{Newton-Okounkov Bodies and Jet Separation: Canonical-Free and Multipoint Generalizations}
\author{Yi Lu\thanks{Yi Lu. Email: ylulu0610@gmail.com}\\
  \small School of Mathematical Sciences, Capital Normal University \\
  \small Department of Mathematical Sciences, University of Liverpool}
 
\date{\today}

\newcommand{\lltriangle}[1]{%
\Delta^{-1}_{#1}
}

\newtheorem{theorem}{Theorem}[section]
\newtheorem{lemma}[theorem]{Lemma}

\newtheorem{corollary}[theorem]{Corollary}
\theoremstyle{definition}
\newtheorem{definition}[theorem]{Definition}

\theoremstyle{remark}
\newtheorem{remark}[theorem]{Remark}
\theoremstyle{conjecture}

\newtheorem*{theorem*}{Theorem}

\crefname{theorem}{Theorem}{Theorems}
\crefname{lemma}{Lemma}{Lemmas}
\crefname{proposition}{Proposition}{Propositions}
\crefname{corollary}{Corollary}{Corollaries}
\crefname{definition}{Definition}{Definitions}
\crefname{example}{Example}{Examples}
\crefname{remark}{Remark}{Remarks}
\crefname{conjecture}{Conjecture}{Conjectures}

\crefformat{equation}{#2(#1)#3}
\Crefformat{equation}{Equation~#2(#1)#3}

\begin{document}
 
\maketitle
% \begin{abstract}
% We establish three generalizations of the Küronya-Lozovanu jet separation criterion using Newton-Okounkov bodies. The result states that if an inverted standard simplex of size \(n+k+\varepsilon\) is contained in all infinitesimal Newton-Okounkov bodies at a point \(x\), then \(K_X + D\) separates \(k\)-jets at \(x\). We prove: (1) a canonical-free version removing the \(K_X\) term at the cost of a computable multiple \(m(D)\); (2) a multipoint extension allowing simultaneous jet separation at finitely many prescribed points; and (3) a combination of both, providing a canonical-free multipoint criterion. Our proofs utilize Trusiani's multipoint Newton-Okounkov body framework combined with Nadel vanishing for asymptotic multiplier ideals. We conclude with explicit computations for a double cover of a product of elliptic curves, following an example of Kollár's, which is an excellent testing ground for the effectiveness of our criteria and for obtaining concrete numerical thresholds that relate the containment parameter to jet separation order.
% \end{abstract}
\begin{abstract}
We establish three generalizations of the Küronya-Lozovanu jet separation criterion via Newton-Okounkov bodies: if an inverted standard simplex of size $n+k+\varepsilon$ is contained in all infinitesimal Newton-Okounkov bodies at $x$, then $K_X+D$ separates $k$-jets at $x$. We prove (1) a canonical-free version with a computable multiple $m(D)$; (2) a multipoint extension for simultaneous jet separation; and (3) a combination of both. Proofs use Trusiani's framework and Nadel vanishing. We conclude with explicit computations for a double cover of a product of elliptic curves.
\end{abstract}

% \tableofcontents
% \newpage 
  % 在这里添加 section 重定义

% ========================================
% ★★★ SECTION 1: INTRODUCTION ★★★
% ========================================
\section{Introduction}
%-背景与研究动机
%--经典的NObody理论：
%Okounkov: \cite{Okounkov1996BrunnMinkowskiIF,  Okounkov2000WhyWM} 
%Lazarsfeld-Mustata: \cite{ASENS_2009_4_42_5_783_0}在这里面也给出infnitesimal Newton Okounkov body的构造
%Kaveh-Khovanskii:\cite{Kaveh2009NewtonOkounkovBS}
%将一个点的Newton-Okounkov body推广到多个点上：
% Shin: \cite{Shin2017ExtendedOB},
% Trusiani:\cite{Trusiani2018MultipointOB}
%--Jet Ampleness与Jet Separation:
% Beltrametti& Sommese关于k-jet ampleness的文献:\citeBeltramettiSommese1993}。有些论文称之为generation of jets/high order embedding。是globally generated = 0-jet， very ample = 1-jet的推广
% 关于jet separation可以在lazasfield Def 5.1.15\cite{lazasfield2004}看到，是jet ampleness在一点处的简化。
%--两个理论交汇点：
%Küronya-Lozovanu的两个工作\cite{KuronyaLozovanu2017}\cite{küronya2015infinitesimalnewtonokounkovbodiesjet}分别用Newton-Okounkov body和infinitesimal Newton-Okounkov body做了nef/ample的结果，其中\cite{küronya2015infinitesimalnewtonokounkovbodiesjet}通过infinitesimal Newton-Okounkov body在一个点上的jet separation。
%一个自然的想法是将一个点的jet separation推广到预先给定多个点上的jet separation的criterion，这也正是本文的主要工作。在文章\cite{Shin2017ExtendedOB}中也给出了多个点的jet separation的判定，然而这篇文章是基于另外一个multipoint Newton-Okounkov body的构造Trusiani:\cite{Trusiani2018MultipointOB}来进行推广的。
\subsection{Background and Motivation}
Throughout this paper, we work over $\mathbb{C}$, and all varieties $X$ are assumed to be irreducible, smooth, projective algebraic varieties of dimension $n$.

The theory of Newton-Okounkov bodies, introduced by \cite{Okounkov1996BrunnMinkowskiIF, Okounkov2000WhyWM}, \cite{ASENS_2009_4_42_5_783_0}, and \cite{Kaveh2009NewtonOkounkovBS}, provides a powerful tool for studying divisors on projective varieties through convex geometry. This construction associates to a divisor a convex body in Euclidean space that encodes asymptotic information about the corresponding linear series. A refinement of this theory is the infinitesimal Newton-Okounkov body, also developed by \cite{ASENS_2009_4_42_5_783_0}, which offers certain advantages such as independence from auxiliary flag choices in favorable situations. The extension from single-point to multipoint configurations has been pursued by \cite{Shin2017ExtendedOB} and \cite{Trusiani2018MultipointOB}, providing finer geometric invariants.

On the algebraic side, the notion of $k$-jet ampleness, systematically studied by \cite{BeltramettiSommese1993}, generalizes classical positivity conditions: $0$-jet ampleness corresponds to global generation, while $1$-jet ampleness corresponds to very ampleness. This concept is also referred to as generation of jets or high-order embeddings in the literature. A local variant is jet separation (see \cite[Definition 5.1.15]{lazasfield2004}), which requires that global sections separate jets at a given point up to a specified order, and can be viewed as a pointwise version of jet ampleness.

The connection between Newton-Okounkov bodies and positivity notions was established by \cite{KuronyaLozovanu2017, küronya2015infinitesimalnewtonokounkovbodiesjet}. In particular, \cite{küronya2015infinitesimalnewtonokounkovbodiesjet} provides a criterion for jet separation at a single point using infinitesimal Newton-Okounkov bodies: the inclusion of an appropriately sized inverted standard simplex guarantees the desired jet separation property.

\begin{theorem}[\cite{küronya2015infinitesimalnewtonokounkovbodiesjet}, Proposition 4.9]
If there exist a positive real number $\varepsilon$ and a natural number $k$ such that 
$$\Delta^{-1}_{n+k+\varepsilon} \subseteq \Delta_{Y'_\bullet}(\pi^*(D))$$
for every infinitesimal flag $Y'_\bullet$ over $x$, where $\Delta^{-1}$ is inverted standard simplex of side length $n+k+\epsilon $. Then $K_X + D$ separates $k$-jets.
\end{theorem}

A natural question is whether this criterion extends to the multipoint setting, where one requires simultaneous jet separation at a prescribed finite set of points. This is the main focus of the present paper. We note that \cite{Shin2017ExtendedOB} also addresses multipoint jet separation criteria, but based on a different construction of multipoint Newton-Okounkov bodies. Our approach instead builds on Trusiani's framework \cite{Trusiani2018MultipointOB}, which is particularly well-suited for infinitesimal constructions.

% 介绍和大猜想之间的关系，引用相关文献。
The motivation for studying multipoint jet separation criteria via Okounkov bodies is closely related to classical questions about effective positivity of line bundles, for example Matsusaka's Big Theorem.

\begin{theorem}[Matsusaka's Big Theorem \cite{MR1898197}]
Let $L$ be an ample divisor and $B$ a nef divisor on a projective variety $X$ of dimension $n$. Set $\rho_L = (L^n)$, $\rho_K = (K_X \cdot L^{n-1})$, and $\rho_B = (B \cdot L^{n-1})$. Then there exists an integer $M_1 = M_1(\rho_L, \rho_K, \rho_B)$, depending only on these intersection numbers and the dimension $n$, such that $mL - B$ is very ample for all $m \geq M_1$.
\end{theorem}

In the toric setting, such positivity questions admit purely combinatorial answers in terms of polytope data: \cite{MR2273280} proved Fujita's very ampleness conjecture for toric varieties using combinatorial properties of associated polytopes, and \cite{MR4297848} characterized $k$-jet ampleness via intersection numbers with invariant curves and lattice lengths of polytope edges. Our work contributes to this philosophy by providing combinatorial criteria for jet separation—a local positivity notion—using the convex geometry of multipoint Newton-Okounkov bodies that determine for which values of $k$ a given line bundle $L$ generates $k$-jets at all prescribed points.

A key distinction of our work is that we provide criteria for $k$-jet ampleness directly in terms of the line bundle $L$ itself, rather than relying on the adjoint bundle $K_X + L$. Since global generation and very ampleness correspond to $0$- and $1$-jet ampleness respectively, our multipoint criterion provides a computational framework for testing such positivity phenomena via convex geometry.

\subsection{Main Results}

Our main results extend the jet separation criterion of \cite{küronya2015infinitesimalnewtonokounkovbodiesjet} in three directions: removing the canonical divisor, allowing multiple points, and combining both generalizations.

\begin{center}
\begin{tabular}{|c|c|c|}
\hline
& \textbf{Single Point} & \textbf{Multiple Points} \\
\hline
\textbf{With $K_X$} & \cite[Proposition 4.9]{küronya2015infinitesimalnewtonokounkovbodiesjet}   & \cref{prop:multipoint-NO-jet-separation} \\
\hline
\textbf{Without $K_X$} & \cref{thm:main_jet_separation} & \cref{thm:generalized_multipoint_jet_separation} \\
\hline
\end{tabular}
\end{center}

\begin{theorem*}[Canonical-Free Jet Separation, cf. \cref{thm:main_jet_separation}]
Let $X$ be an $n$-dimensional smooth projective variety over $\mathbb{C}$, $D$ an ample divisor on $X$, and $x \in X$ a point. Let $\pi: \overline{X} \to X$ be the blow-up at $x$, and let $\tilde{Y}_{\bullet}$ be an infinitesimal flag centered at $x$. There exists $m(D)$ (depending only on $D$) such that $m(D) \cdot D - K_X$ is ample. If for some positive integer $m$, there exist $\varepsilon > 0$ and $k \in \mathbb{N}$ such that 
$$\Delta^{-1}_{n+k+\varepsilon} \subseteq \Delta_{\tilde{Y}_{\bullet}}(m\pi^*(D))$$
for every infinitesimal flag $\tilde{Y}_{\bullet}$ over $x$, then $(m + m(D)) \cdot D$ separates $k$-jets at $x$.
\end{theorem*}

The above result addresses jet separation at a single point. We now extend this to the multipoint setting.

For a big divisor $D$ on $X$ and a 0-cycle $Z = \{x_1, \ldots, x_N\}$ of distinct smooth points, let $\pi: \overline{X} = \mathrm{Bl}_{Z}(X) \to X$ be the blow-up along $Z$. For each point $x_j$, let $\tilde{Y}^{(j)}_{\bullet}$ be an infinitesimal flag over $x_j$, and denote ${\tilde{Y}_{\bullet}}=({\tilde{Y}_\bullet^{(1)}},...,{\tilde{Y}_\bullet^{(N)}})$. Define
\[
V_{k,j} = \{ s \in H^0(X, kD) \setminus \{0\} : \nu_{\tilde{Y}^{(j)}_\bullet}(s) < \nu_{\tilde{Y}^{(i)}_\bullet}(s) \text{ for all } i \neq j \}.
\]
The \emph{multipoint infinitesimal Newton-Okounkov body} at $x_j$ is
\[
\Delta_{\tilde{Y}_\bullet,j}(D) = \overline{\bigcup_{k \geq 1} \frac{\nu_{\tilde{Y}^{(j)}_\bullet}(V_{k,j})}{k}} \subset \mathbb{R}^n
.
\] 

With this notion, we obtain the following multipoint jet separation criterion involving the canonical divisor.

\begin{theorem*}[Multipoint Jet Separation, cf. \cref{prop:multipoint-NO-jet-separation}]
Let $X$ be an $n$-dimensional smooth projective variety, $D$ an ample divisor, and $Z = \{p_1, \ldots, p_N\}$ a 0-cycle of distinct smooth points on $X$. Let $\pi: \overline{X} = \mathrm{Bl}_{Z}(X) \to X$ be the blow-up along $Z$. If there exist $\varepsilon > 0$ and $k \in \mathbb{N}$ such that 
$$\Delta_{n+k+\varepsilon}^{-1} \subseteq \Delta_{\tilde{Y}_\bullet,j}(\pi^*(D))$$
for all $j = 1, \ldots, N$ and every tuple of infinitesimal flags $\tilde{Y}_\bullet = (\tilde{Y}^{(1)}_\bullet, \ldots, \tilde{Y}^{(N)}_\bullet)$ with $\tilde{Y}^{(i)}_\bullet$ over $p_i$, then $K_X + D$ separates $k$-jets at $p_1, \ldots, p_N$ simultaneously.
\end{theorem*}

Finally, we present the canonical-free version for the multipoint case, which parallels the single-point result.

\begin{theorem*}[Canonical-Free Multipoint Jet Separation, cf. \cref{thm:generalized_multipoint_jet_separation}]
Let $X$ be an $n$-dimensional smooth projective variety over $\mathbb{C}$, $D$ an ample divisor on $X$, and $Z = \{p_1, \ldots, p_N\}$ a 0-cycle of distinct smooth points on $X$. There exists $m(D)$ (depending only on $D$) such that $m(D) \cdot D - K_X$ is ample. If for some positive integer $m$, there exist $\varepsilon > 0$ and $k \in \mathbb{N}$ such that 
$$\Delta_{n+k+\varepsilon}^{-1} \subseteq \Delta_{\tilde{Y}_\bullet,j}(m\pi^*(D))$$
for all $j = 1, \ldots, N$ and every tuple of infinitesimal flags $\tilde{Y}_\bullet = (\tilde{Y}^{(1)}_\bullet, \ldots, \tilde{Y}^{(N)}_\bullet)$ with $\tilde{Y}^{(i)}_\bullet$ centered at $p_i$, then $(m + m(D)) \cdot D$ separates $k$-jets at $p_1, \ldots, p_N$ simultaneously.
\end{theorem*}

\begin{corollary}
    Let $X,D$ and $m(D)$ be as above and $k\in \mathbb{N}$. If for some positive integer in there exist $e>0$ such that $$\Delta_{n+k+\varepsilon}^{-1} \subseteq \Delta_{\tilde{Y}_\bullet,j}(m\pi^*(D))$$
    for all points $p_1,...,p_{k+1}\in X,j=1,...,k+1$ and every tuple of infinitesimal flags $\tilde{Y}_\bullet$, then $(m+m(D))\cdot D$ is $k-$jet ample
\end{corollary}

We conclude with an explicit computation based on an example originally due to \cite[Example 3.7]{EinLazarsfeld1993}, who considered a double covering of a product of two elliptic curves to exhibit a sequence of ample divisors whose multiples required larger and larger multiples to become very ample. We extend this phenomenon to k-jet ampleness and:

\begin{enumerate} 

\item Computing explicitly the threshold $k$ such that $(m+k)D_n$ is $m$-jet ample on the double cover, where $D_n = f^* A_n$. This provides concrete numerical bounds for higher-order jet separation and illustrates how the required multiples grow with the jet order.

\item Calculating the Newton-Okounkov body $\Delta(D_n)$ with respect to an admissible flag. We emphasize that our computation uses a standard flag construction; the infinitesimal flag approach, while theoretically more refined, presents significant computational challenges in this setting.
\end{enumerate}

This calculation demonstrates the effectiveness of our Newton-Okounkov body approach and provides concrete numerical bounds for jet separation in a non-trivial geometric setting.

\section*{Acknowledgements}

I gratefully acknowledge the financial support from the China Scholarship Council (CSC) under the grant number [202307300052].

I would like to express my sincere gratitude to my advisor, Dr.~Thomas Eckl, for proposing the research problem and providing invaluable guidance throughout this work. I am deeply grateful to my co-advisor, Dr.~Zhixian Zhu, for his theoretical support on toric varieties and meticulous verification of technical details in this paper.

This research was conducted under a joint supervision program.

% ========================================
% ★★★ SECTION 2: DEFINITIONS AND PRELIMINARIES ★★★
% ========================================
\section{Definitions and Preliminaries}
\label{sec:preliminaries}
% --- 2.1 Newton-Okounkov凸体 ---
% • 单点情形的精确定义与构造
% • 多点Newton-Okounkov凸体的定义
% • 多点情形的构造方式与赋值选择
% • 基本性质比较
% • 经典例子
\subsection{Newton-Okounkov Bodies}

\begin{definition}[Newton-Okounkov Body]\label{def:NewtonOkounkovBody}
Let $X$ be an irreducible projective variety of dimension $n$, and fix an admissible flag
\[Y_\bullet: X = Y_0 \supseteq Y_1 \supseteq Y_2 \supseteq \ldots \supseteq Y_{n-1} \supseteq Y_n = \{pt\},\]
where $\operatorname{codim}_X(Y_i) = i$ and each $Y_i$ is non-singular at the point $Y_n$.

Given a big divisor $D$ on $X$, we define a valuation-like function
\[\nu_{Y_\bullet}: H^0(X, \mathcal{O}_X(D)) - \{0\} \rightarrow \mathbb{Z}^n\]
by setting $\nu_1(s) = \operatorname{ord}_{Y_1}(s)$, and inductively defining $\nu_i(s)$ as the order of vanishing along $Y_i$ of the restriction of $s$ to $Y_{i-1}$ after dividing by the appropriate power of a local equation of $Y_{i-1}$.

The graded semigroup of $D$ is defined as:
\[\Gamma(D) = \{(\nu_{Y_\bullet}(s), m) \mid 0 \neq s \in H^0(X, \mathcal{O}_X(mD)), m \geq 0\} \subset \mathbb{N}^n \times \mathbb{N}\]

The Newton-Okounkov body of $D$ is then:
\[\Delta(D) = \Delta_{Y_\bullet}(D) = \Sigma(\Gamma) \cap (\mathbb{R}^n \times \{1\})\]
where $\Sigma(\Gamma)$ is the closed convex cone spanned by $\Gamma(D)$.
\end{definition}

\begin{remark}
The construction of the NO-body does not require the choice of a flag. A flag is merely one means to obtain a valuation $\nu: K(X)^*\rightarrow \mathbb{Z}^n$ of rank $n$, from which we define the NO-body $\Delta(D,\nu)$ via the associated value semigroup. Such valuations can be constructed in many ways. This valuation-theoretic perspective, developed systematically in \cite{Kaveh2009NewtonOkounkovBS} and further extended in \cite{Kaveh2016KhovanskiiBH}, provides a more intrinsic framework that encompasses the flag construction as a special case.

A notable example is the jet-valuation $\nu_{jet}$, which arises from applying an invertible linear transformation to the valuation vector associated with an infinitesimal flag. This transformation, known as jet separation, establishes a natural correspondence between NO-bodies with respect to $\nu_{jet}$ and those with respect to the infinitesimal flag. More precisely, if the jet valuation is given by \[\nu_{jet} = (\nu_1, \nu_2, \ldots, \nu_k),\] then the infinitesimal flag valuation is defined by the transformation
\[\nu_{inf} = (\nu_1+\nu_2+...+\nu_k, \nu_2, \nu_3 , \ldots, \nu_k).\]
This linear change of coordinates induces a corresponding transformation between the associated NO-bodies, as illustrated in Figure~\ref{fig:jet_transformation}.
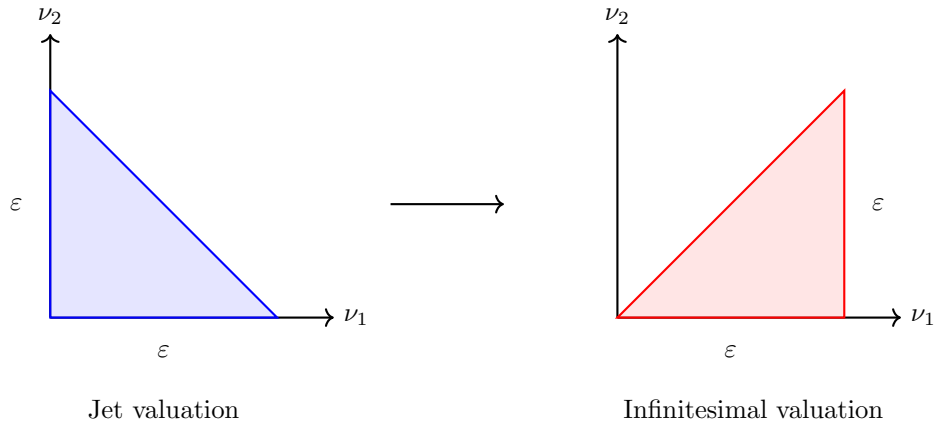
\begin{figure}[h]
\centering
\begin{tikzpicture}[scale=1.5]
% Left triangle (infinitesimal flag valuation)
\draw[thick, ->] (0,0) -- (2.5,0) node[right] {\(\nu_1\)};
\draw[thick, ->] (0,0) -- (0,2.5) node[above] {\(\nu_2\)};
\fill[blue!20, opacity=0.5] (0,0) -- (2,0) -- (0,2) -- cycle;
\draw[thick, blue] (0,0) -- (2,0) -- (0,2) -- cycle;
% Label edges with epsilon
\node at (1, -0.3) {\(\varepsilon\)};
\node at (-0.3, 1) {\(\varepsilon\)};
\node at (1, -0.8) {Jet valuation};

% Arrow indicating transformation
\draw[->, thick] (3,1) -- (4,1) ;

% Right triangle (jet valuation)
\begin{scope}[xshift=5cm]
\draw[thick, ->] (0,0) -- (2.5,0) node[right] {\(\nu_1\)};
\draw[thick, ->] (0,0) -- (0,2.5) node[above] {\(\nu_2\)};
\fill[red!20, opacity=0.5] (0,0) -- (2,0) -- (2,2) -- cycle;
\draw[thick, red] (0,0) -- (2,0) -- (2,2) -- cycle;
% Label edges with epsilon
\node at (1, -0.3) {\(\varepsilon\)};
\node at (2.3, 1) {\(\varepsilon\)};
\node at (1.2, -0.8) {Infinitesimal valuation};
\end{scope}
\end{tikzpicture}
\caption{Transformation between jet valuation and infinitesimal valuation: \(\triangle_{\varepsilon} \rightarrow \Delta^{-1}_\epsilon\)}
\label{fig:jet_transformation}
\end{figure}

\end{remark}

\begin{definition}[Infinitesimal Newton-Okounkov Body]\label{def:InfinitesimalNewtonOkounkovBody}
Let $X$ be an irreducible projective variety of dimension $n$. Fix a smooth point $x \in X$ and a complete flag $V_\bullet$ of subspaces
\[T_xX = V_0 \supseteq V_1 \supseteq V_2 \supseteq \ldots \supseteq V_{n-1} \supseteq \{0\}\]
in the tangent space to $X$ at $x$.

Consider the blowing up $\pi: \overline{X} = \text{Bl}_x(X) \rightarrow X$ of $X$ at $x$, with exceptional divisor $E = \mathbb{P}(T_xX)$. The projectivizations of the $V_i$ give rise to a flag $\tilde{Y}_{\bullet} = F(x; V_\bullet)$ of subvarieties of $\overline{X}$.

For a big divisor $D$ on $X$, we define $D' = \pi^*D$. Since $H^0(X, \mathcal{O}_X(mD)) = H^0(\overline{X}, \mathcal{O}_{\overline{X}}(mD'))$ for all $m$, we define:
\[\Delta_{\tilde{Y}_{\bullet}}(D) = \Delta_{\tilde{Y}_{\bullet}}(D')\]

For very general choices of $x$ and $V_\bullet$, the Okounkov bodies $\Delta_{\tilde{Y}_{\bullet}}(D) \subset \mathbb{R}^n$ all coincide. This gives rise to a canonically defined convex body $\Delta'(D) \subset \mathbb{R}^n$ that does not depend on any auxiliary choices.
\end{definition}

\begin{remark}\label{rem:infinitesimal_valuation}
In \cite{küronya2015infinitesimalnewtonokounkovbodiesjet}, the infinitesimal valuation $\nu_{\tilde{Y}_{\bullet}}: H^0(X, \mathcal{O}_X(D)) \setminus \{0\} \to \mathbb{Z}^n$ has the following structure: the first coordinate records the minimal vanishing order along $\tilde{Y}_1$, while the remaining coordinates encode the valuation of the lexicographically minimal monomial among all lowest-order terms. This is because the quotient by a local equation for $\tilde{Y}_1$ eliminates higher-order terms in the first step, leaving only the minimal monomial for subsequent valuation computation.
\end{remark}

\begin{definition}[Inverted Standard Simplex]
For a positive real number \(\xi \geq 0\), the inverted standard simplex of size \(\xi\), denoted by \(\Delta_{\xi}^{-1}\), is the convex hull of the set
\[
\Delta_{\xi}^{-1} \stackrel{\text{def}}{=} \{0, \xi e_1, \xi(e_1 + e_2), \ldots, \xi(e_1 + \cdots + e_n)\} \subseteq \mathbb{R}^n.
\]
When \(\xi = 0\), then \(\Delta_{\xi}^{-1} = \mathbf{0}\).
\end{definition}

% \begin{theorem}[\cite{küronya2015infinitesimalnewtonokounkovbodiesjet} Theorem 3.1]
% \label{thm:stable-base-locus-characterization}
% Let $X$ be a smooth projective variety, $D$ a big $\mathbb{R}$-divisor and $x \in X$ an arbitrary point on $X$. Let $\mathbf{B}_-(D)$ be restricted base loci. Then the following are equivalent.
% \begin{enumerate}
% \item[(1)] $x \notin \mathbf{B}_-(D)$.
% \item[(2)] There exists an infinitesimal flag $Y_{\bullet}$ over $x$ such that $\mathbf{0} \in \Delta_{Y_{\bullet}}(D)$.
% \item[(3)] For every infinitesimal flag $Y_{\bullet}$ over $x$, one has $\mathbf{0} \in \Delta_{Y_{\bullet}}(D)$.
% \end{enumerate}
% \end{theorem}
\begin{theorem}[\cite{KuronyaLozovanu2017} Theorem 2.1]
\label{thm:stable-base-locus-characterization}
Let $D$ be a big $\mathbb{R}$-divisor on a smooth projective variety $X$ of dimension $n$, let $x \in X$. Then the following are equivalent.
\begin{enumerate}
    \item[(1)] $x \notin \mathbf{B}_-(D)$.
    \item[(2)] There exists an admissible flag $Y_\bullet$ on $X$ centered at $x$ such that $\mathbf{0} \in \Delta_{Y_\bullet}(D) \subseteq \mathbb{R}^n$.
    \item[(3)] The origin $\mathbf{0} \in \Delta_{Y_\bullet}(D)$ for every admissible flag $Y_\bullet$ on $X$ centered at $x \in X$.
\end{enumerate}
\end{theorem}

The construction of multipoint Newton-Okounkov bodies can be found in \cite{Trusiani2018MultipointOB}.
\begin{definition}[Multipoint Newton-Okounkov Body]\label{def:MultipointNewtonOkounkovBody}
Let $X$ be an irreducible projective variety of dimension $n$, and fix $N$ distinct points \(x_1, \ldots, x_N \in X\) with corresponding faithful valuations \(\nu_{x_1}, \ldots, \nu_{x_N}\).

For a big divisor $D$ on $X$ (or equivalently, a big line bundle $L = \mathcal{O}_X(D)$), we define the subspace $V_{k,j} \subset H^0(X, kL)$ as
\[V_{k,j} = \{s \in H^0(X, kL) \setminus \{0\} : \nu_{x_j}(s) < \nu_{x_i}(s) \text{ for any } i \neq j\}.\]

The graded semigroup at point $x_j$ is defined as:
\[\Gamma_j(L) = \{(\nu_{x_j}(s), k) : s \in V_{k,j}, k \geq 1\} \subset \mathbb{Z}^n \times \mathbb{Z}\]

The multipoint Newton-Okounkov body of $L$ at $x_j$ is then:
\[\Delta_j(L) = \Delta(\Gamma_j) = \overline{\bigcup_{k\geq 1}\frac{\nu_{x_j}(V_{k,j})}{k}} \subset \mathbb{R}^n
\]

These multipoint Okounkov bodies depend on the choice of the faithful valuations \(\nu_{x_1}, \ldots, \nu_{x_N}\), but we omit this dependence to simplify notation.
\end{definition}

An explicit construction for toric varieties is given in \cite[Theorem 6.4]{Trusiani2018MultipointOB}.

\begin{definition}[Multipoint Infinitesimal Newton-Okounkov Body]\label{def:MultipointInfinitesimalNewtonOkounkovBody}
Let $X$ be an irreducible projective variety of dimension $n$, and let $Z = \{x_1, \ldots, x_N\} \subset X$ be a 0-cycle of distinct smooth points. For each point \(x_j\), fix a complete flag of subspaces
\[T_{x_j}X = V^{(j)}_0 \supseteq V^{(j)}_1 \supseteq \ldots \supseteq V^{(j)}_{n-1} \supseteq \{0\}.\]

Consider the blow-up \(\pi: \overline{X} = \mathrm{Bl}_{Z}(X) \to X\) along the 0-cycle $Z$ with exceptional divisors \(E_j = \mathbb{P}(T_{x_j}X)\) over each point $x_j$. The projectivizations of the \(V^{(j)}_i\) give rise to flags \(\tilde{Y}^{(j)}_{\bullet} = F(x_j; V^{(j)}_\bullet)\) of subvarieties of \(\overline{X}\). And we have ${\tilde{Y}_{\bullet}}=({\tilde{Y}_\bullet^{(1)}},...,{\tilde{Y}_\bullet^{(N)}})$.

For a big divisor $D$ on $X$, define $D' = \pi^*D$. Since $H^0(X, \mathcal{O}_X(mD)) = H^0(\overline{X}, \mathcal{O}_{\overline{X}}(mD'))$ for all $m$, we define the valuation:
\[\nu_{\tilde{Y}^{(j)}_{\bullet}}(s) = \nu_{\tilde{Y}^{(j)}_{\bullet}}(\pi^*s)\]

Define the subspaces
\[
V_{k,j} = \{ s \in H^0(X, kD) \setminus \{0\} : \nu_{\tilde{Y}^{(j)}_\bullet}(s) < \nu_{\tilde{Y}^{(i)}_\bullet}(s) \text{ for all } i \neq j \}.
\]

The graded semigroup at \(x_j\) is
\[
\Gamma_{\tilde{Y}_\bullet,j}(D) = \{ (\nu_{\tilde{Y}^{(j)}_\bullet}(s), k) : s \in V_{k,j}, k \geq 1 \} \subset \mathbb{Z}^n \times \mathbb{Z}.
\]

The multipoint infinitesimal Newton-Okounkov body at \(x_j\) is defined as
\[
\Delta_{\tilde{Y}_\bullet,j}(D) = \Delta(\Gamma_{\tilde{Y}_\bullet,j}) = \overline{\bigcup_{k \geq 1} \frac{\nu_{\tilde{Y}^{(j)}_\bullet}(V_{k,j})}{k}} \subset \mathbb{R}^n
.
\] 
\end{definition}

\begin{lemma}\label{lem:InclusionInSinglePointBody}
For each $j \in \{1, \ldots, N\}$, we have
\[
\Delta_{\tilde{Y}_\bullet,j}(D) \subseteq \Delta_{\tilde{Y}^{(j)}_\bullet}(D),
\]
where $\Delta_{\tilde{Y}_\bullet,j}(D)$ is the multipoint infinitesimal Newton-Okounkov body at $x_j$ defined in Definition~\ref{def:MultipointInfinitesimalNewtonOkounkovBody}, and $\Delta_{\tilde{Y}^{(j)}_\bullet}(D)$ is the classical infinitesimal Newton-Okounkov body with respect to the single flag $\tilde{Y}^{(j)}_\bullet$ at $x_j$.
\end{lemma}

\begin{proof}
This follows immediately from the construction, since $V_{k,j} \subseteq H^0(X, kD) \setminus \{0\}$.
\end{proof}

Let $D_j = Y_{j,1}$ be prime divisors and $\mathbb{D} := \sum_{i=1}^N D_i$. Define 
\[
\mu(L; \mathbb{D}) := \sup\{t \geq 0 : L - t\mathbb{D} \text{ is big}\}, \quad \mu(L; D_j) := \sup\{t \geq 0 : \Delta_j(L - t\mathbb{D})^{\circ} \neq \emptyset\}.
\]
\begin{theorem}[\cite{Trusiani2018MultipointOB} Theorem 3.21]\label{thm:okounkov-body-slice}
Let $L$ be a big $\mathbb{R}$-line bundle, $\nu_{p_j}$ a family of valuations associated to a family of admissible flags $Y_{\cdot}$ centered at $p_1, \ldots, p_N$. 
Then, letting $(x_1, \ldots, x_n)$ be fixed coordinates on $\mathbb{R}^n$, for any $j \in \{1, \ldots, N\}$ such that $\Delta_j(L)^{\circ} \neq \emptyset$, we have

$$\Delta_j(L)_{x_1 \geq t} = \Delta_j(L - t\mathbb{D}) + t\vec{e}_1$$ for any $0 \leq t < \mu(L; D_j)$, where $\Delta_j(L)_{x_1 \geq t}$ denotes the slice where the first coordinate is at least $t$.
\end{theorem}

% --- 2.2 Jet分离理论 ---
% • Jet概形的定义
% • Jet分离的定义与几何解释
% • 与乘子理想的联系 
\subsection{Jet Separation and Jet Ampleness}
\begin{definition}[Separation of jets]
Let $L$ be a line bundle on a projective variety $X$, and fix a smooth point $x \in X$ with maximal ideal $\mathfrak{m}_x \subseteq \mathcal{O}_X$. One says that $|L|$ \emph{separates $s$-jets at $x$} if the natural map
\[
H^0(X,L) \longrightarrow H^0\bigl(X, L \otimes \mathcal{O}_{X,x_i}/ \mathfrak{m}_x^{s+1}\bigr) 
\]
taking sections of $L$ to their $s$-jets is surjective.

For a divisor $D$ on $X$, we say that $D$ \emph{separates $s$-jets at $x$} if the line bundle $\mathcal{O}_X(D)$ separates $s$-jets at $x$.

Thus $L$ (or $D$) separates $0$-jets if and only if $L$ (or $\mathcal{O}_X(D)$) is free at $x$, and in this case $L$ (or $D$) separates $1$-jets if and only if the derivative $d_x \phi_{|L|}$ (or $d_x \phi_{|D|}$) at $x$ of the regular map is injective.
\end{definition}

\begin{definition}[k-Jet Ampleness]
Let $X$ be an $n$-dimensional projective manifold and $L$ a line bundle on $X$. For a nonnegative integer $k$, we say that $L$ is $k$-jet ample if, given any $N$ integers $k_1, \ldots, k_N$ such that $k + 1 = \sum_{i=1}^N k_i$, and any $N$ distinct points $\{x_1, \ldots, x_N\} \subset X$, the evaluation map
$$H^0(X, L) \to H^0\bigl(X, L \otimes \mathcal{O}_{X}/(\otimes_{i=1}^N\mathfrak{m}_{x_i}^{k_i})\bigr)\to 0$$
is surjective, where $\mathfrak{m}_{x_i}$ denotes the maximal ideal at $x_i$.

For a divisor $D$ on $X$, we say that $D$ is $k$-jet ample if the line bundle $\mathcal{O}_X(D)$ is $k$-jet ample.

Note that:
\begin{itemize}
\item $L$ (or $D$) is spanned (i.e., globally generated) if and only if it is $0$-jet ample
\item $L$ (or $D$) is very ample if and only if it is $1$-jet ample
\end{itemize}
\end{definition}

\begin{remark}
Jet ampleness can be viewed as a generalization of jet separation to multiple points. It requires that global sections can simultaneously realize any prescribed jet data at any finite collection of points, with the total order of jets constrained by the ampleness parameter $k$.
\end{remark}

 \begin{definition}[k-Jet Ampleness Supported on a 0-cycle]\label{def:kJetAmplenessSupportedOn0Cycle}
Let $X$ be an $n$-dimensional projective manifold, $L$ a line bundle on $X$, and let $Z = \{x_1, \ldots, x_N\} \subset X$ be a 0-cycle of distinct points. For a nonnegative integer $k$, we say that $L$ is \emph{$k$-jet ample supported on $Z$} if, for any integers $k_1, \ldots, k_N$ satisfying
\[
k + 1 = \sum_{j=1}^N k_j,
\]
the evaluation map
\[
 H^0(X, L) \to H^0\bigl(X, L \otimes \mathcal{O}_{X}/(\otimes_{i=1}^N\mathfrak{m}_{x_i}^{k_i})\bigr)\to 0
\]
is surjective, where $m_{x_j}$ denotes the maximal ideal at $x_j$.

For a divisor $D$ on $X$, $D$ is said to be $k$-jet ample supported on $Z$ if $\mathcal{O}_X(D)$ is $k$-jet ample supported on $Z$.
\end{definition}

\begin{theorem}[\cite{cyclic_higher_order_embed} Theorem 2.1]\label{thm:cyclic_covering_jet_ample}
Let $X$ be a smooth projective variety and $B \subset X$ a smooth divisor. Let $M$ be a line bundle on $X$ such that $\mathcal{O}_X(dM) \cong \mathcal{O}_X(B)$ and let $\pi : Y \to X$ be the cyclic covering of degree $d$ defined by $M$. Let $L$ be a line bundle on $X$ and $k$ a non-negative integer. If $L - qM$ is $(k-q)$-jet ample for $q = 0, \ldots, \min(k, d-1)$, then $\pi^*L$ is $k$-jet ample.
\end{theorem}

% ========================================
% ★★★ SECTION 3: MAIN THEOREMS ★★★
% ========================================
\section{Main Theorems}
\label{sec:main}
% --- 3.1单点Jet分离理论的推广 ---
% • 建立Newton-Okounkov体与Jet分离的桥梁
% • 提供基于凸几何的Jet性质判断方法
% • 深化代数几何中局部性质的理解
\begin{theorem}[Canonical-Free Jet Separation]\label{thm:main_jet_separation}
Let $X$ be an $n$-dimensional smooth projective variety over $\mathbb{C}$, let $D$ be an ample divisor on $X$, and let $x \in X$ be a point. Let $\overline{X}$ be the blow-up of $X$ at $x$ with blowing-up map $\pi: \overline{X} \rightarrow X$, and let $\tilde{Y}_{\bullet}$ be an infinitesimal flag centered at $x$ with associated valuation $\nu_{\tilde{Y}_{\bullet}}$. There exists $m(D)$ (depending only on $D$) such that $m(D) \cdot D - K_X$ is ample, and for any positive integer $m$ such that 
$$\lltriangle{n+k+\varepsilon} \subseteq \Delta_{\tilde{Y}_{\bullet}}(m\pi^*(D))$$
for some $\varepsilon > 0$ and natural number $k$, the divisor $(m + m(D)) \cdot D$ separates $k$-jets at $x$.

\end{theorem}

\begin{proof}
This is a special case of Theorem \ref{thm:generalized_multipoint_jet_separation} with $N = 1$. The complete proof will be given.
\end{proof}

\begin{remark}\label{rem:converse_jet_separation}
Based on the proof technique of \cite[Proposition 4.10]{küronya2015infinitesimalnewtonokounkovbodiesjet}, we can derive a converse result: 

If $ D$ separates $k$-jets at $x$, then there exists $\varepsilon > 0$ such that
\[
\lltriangle{n+k} \subseteq \Delta_{\tilde{Y}_{\bullet}}(\pi^*(D)).
\]

This provides a partial converse to Theorem \ref{thm:main_jet_separation}, establishing a more complete correspondence between jet separation properties and the geometric structure of infinitesimal Newton-Okounkov bodies.
\end{remark}

The following theorem adopts the proof strategy from 
\cite[Proposition 4.9]{küronya2015infinitesimalnewtonokounkovbodiesjet}.

\begin{theorem}[Multipoint Jet Separation]
\label{prop:multipoint-NO-jet-separation}
Let $X$ be an $n$-dimensional smooth projective variety, $D$ a big Cartier divisor, and $Z = \{p_1, \ldots, p_N\}$ be a 0-cycle of distinct smooth points on $X$. Consider the blow-up $\pi: \overline{X} = \mathrm{Bl}_{Z}(X) \to X$ along $Z$. Assume that there exists a positive real number $\varepsilon$ and a natural number $k$ with the property that $\Delta_{n+k+\varepsilon}^{-1} \subseteq \Delta_{\tilde{Y}_\bullet,j}(\pi^*(D))$ for all $j = 1, \ldots, N$ and every tuple of infinitesimal flags $\tilde{Y}_\bullet = (\tilde{Y}^{(1)}_\bullet, \ldots, \tilde{Y}^{(N)}_\bullet)$ with $\tilde{Y}^{(i)}_\bullet$ over $p_i$. Then $K_X + D$ separates $k$-jets at the points $p_1, \ldots, p_N$ simultaneously.
\end{theorem}
\begin{proof}
\textbf{Strategy.} We reduce the multipoint jet separation problem to showing surjectivity of a restriction map on the blow-up. Using the slice formula for multipoint Newton-Okounkov bodies and the assumption on simplex inclusions, we verify that an auxiliary divisor $B$ has the origin in all relevant infinitesimal Newton-Okounkov bodies. This ensures the asymptotic base locus avoids the exceptional divisors, allowing us to apply Nadel vanishing to conclude surjectivity.

\textbf{Step 1: Reformulation as a restriction map.}
What we need to prove is that the restriction map
\[
\begin{aligned}
H^0(X, \mathcal{O}_X(K_X + D)) \longrightarrow &H^0\left(X, \mathcal{O}_X(K_X + D) \otimes \mathcal{O}_{X}/(\otimes_{i=1}^{N} \mathfrak{m}_{X,p_i}^{k+1})\right)\\
&=\bigoplus_{i=1}^{N} H^0\left(X, \mathcal{O}_X(K_X + D) \otimes \mathcal{O}_{X,p_i}/\mathfrak{m}_{X,p_i}^{k+1}\right)
\end{aligned}
\]
is surjective.

\textbf{Step 2: Transfer to the blow-up.}
Transferring the question to the blow-up $\overline{X}$, where $E_i$ is the exceptional divisor over $p_i$ and $\mathbb{E}=\sum_{i=1}^NE_i$, this is equivalent to requiring
\begin{equation}\label{eq:restriction-map}
H^0(\overline{X}, \mathcal{O}_{\overline{X}}(\pi^*(K_X + D))) \longrightarrow H^0\left(\overline{X}, \mathcal{O}_{\overline{X}}(\pi^*(K_X + D)) \otimes \mathcal{O}_{\overline{X}}/(\otimes_{i=1}^{N}\mathcal{O}_{\overline{X}}(-(k+1)E_i)))\right)
\end{equation}
to be surjective.

\textbf{Step 3: Construction of auxiliary divisor $B$.}
Let us write $B \stackrel{\text{def}}{=} \pi^*(D) - (n+k)\left(\sum_{i=1}^NE_i\right)$. By \cref{thm:okounkov-body-slice}, we have
\[
\Delta_{\tilde{Y}^{(j)}_\bullet}(B) = \Delta_{\tilde{Y}^{(j)}_\bullet}(\pi^*(D))_{x_0 \geqslant n+k} - (n+k, 0, \ldots, 0),\quad \forall j=1,\ldots,N
\]
for any infinitesimal flag $\tilde{Y}^{(j)}_\bullet$ over the point $p_j$.

\textbf{Step 4: Verification that the origin lies in $\Delta_{\tilde{Y}^{(j)}_\bullet}(B)$.}
The assumption $\Delta_{n+k+\varepsilon}^{-1} \subseteq \Delta_{\tilde{Y}_\bullet,j}(\pi^*(D))$ with $\varepsilon > 0$ ensures that $(n+k, 0, \ldots, 0) \in \Delta_{\tilde{Y}_\bullet,j}(\pi^*(D))$. By \cref{lem:InclusionInSinglePointBody}, $\Delta_{\tilde{Y}_\bullet,j}(\pi^*(D)) \subseteq \Delta_{\tilde{Y}^{(j)}_\bullet}(\pi^*(D))$ for any infinitesimal flag $\tilde{Y}^{(j)}_\bullet$ at $p_j$. Therefore $(n+k, 0, \ldots, 0) \in \Delta_{\tilde{Y}^{(j)}_\bullet}(\pi^*(D))$, and after the translation by $-(n+k, 0, \ldots, 0)$ in the slice formula, we obtain $\mathbf{0} \in \Delta_{\tilde{Y}^{(j)}_\bullet}(B)$ for all $j$. In particular, $B$ is a big line bundle with the property that the origin $\mathbf{0} \in \Delta_{\tilde{Y}^{(j)}_\bullet}(B)$ for any infinitesimal flag $\tilde{Y}^{(j)}_\bullet$ over $p_j$.

\textbf{Step 5: Asymptotic base locus avoids exceptional divisors.}
As a consequence of \cref{thm:stable-base-locus-characterization}, we obtain that $\mathbf{B}_-(B) \cap E_j = \varnothing$ for all $j=1,\ldots,N$. Thus 
\[
\mathrm{Zeroes}(\mathscr{J}(\overline{X}, ||B||)) \cap E_j = \varnothing,\quad \forall j=1,\ldots,N.
\]

\textbf{Step 6: Application of Nadel vanishing.}
Recall that $B = \pi^* D - (n+k)\left(\sum_{i=1}^NE_i\right)$, and $K_{\overline{X}} = \pi^* K_X + (n-1)\left(\sum_{i=1}^NE_i\right)$, therefore we have the short exact sequence
\[
0 \to \mathcal{O}_{\overline{X}}(K_{\overline{X}} + B) \otimes \mathscr{J}(\overline{X}, ||B||) \to \mathcal{O}_{\overline{X}}(\pi^*(K_X + D)) \to \mathcal{O}_{\overline{X}}(\pi^*(K_X + D)) \otimes \left(\mathscr{Z} \oplus \bigoplus_{i=1}^{N}\mathcal{O}_{(k+1)E_i}\right) \to 0,
\]
where $\mathscr{Z}$ stands for the structure sheaf determined by the closed subscheme associated to the ideal $\mathscr{J}(\overline{X}, ||B||)$; note that this latter has support disjoint from $\mathbb{E}$.

Since $B$ is a big line bundle, by Nadel's vanishing for asymptotic multiplier ideals we have
\[
H^1(\overline{X}, \mathcal{O}_{\overline{X}}(K_{\overline{X}} + B) \otimes \mathscr{J}(\overline{X}, ||B||)) = 0.
\]

\textbf{Step 7: Surjectivity of the restriction map.}
Therefore the restriction map
\[
H^0(\overline{X}, \mathcal{O}_{\overline{X}}(\pi^*(K_X + D))) \longrightarrow H^0\left(\overline{X}, \mathcal{O}_{\overline{X}}(\pi^*(K_X + D)) \otimes \left(\mathscr{Z} \oplus \bigoplus_{i=1}^{N}\mathcal{O}_{(k+1)E_i}\right)\right)
\]
is surjective, but then so is
\[
H^0(\overline{X}, \mathcal{O}_{\overline{X}}(\pi^*(K_X + D))) \longrightarrow H^0\left(\overline{X}, \mathcal{O}_{\overline{X}}(\pi^*(K_X + D)) \otimes \left(\bigoplus_{i=1}^{N}\mathcal{O}_{(k+1)E_i}\right)\right),
\]
as required.
\end{proof}

\begin{theorem}[Canonical-Free Multipoint Jet Separation]
\label{thm:generalized_multipoint_jet_separation}
Let $X$ be an $n$-dimensional smooth projective variety over $\mathbb{C}$, let $D$ be an ample divisor on $X$, and let $Z = \{p_1, \ldots, p_N\}$ be a 0-cycle of distinct smooth points on $X$. Consider the blow-up $\pi: \overline{X} = \mathrm{Bl}_{Z}(X) \to X$ along $Z$. 

There exists $m(D)$ (depending only on $D$) such that $m(D) \cdot D - K_X$ is ample. For any positive integer $m$ with the property that 
$$\Delta_{n+k+\varepsilon}^{-1} \subseteq \Delta_{\tilde{Y}_\bullet,j}(m\pi^*(D))$$
for some $\varepsilon > 0$, natural number $k$, and for all $j = 1, \ldots, N$ and every tuple of infinitesimal flags $\tilde{Y}_\bullet = (\tilde{Y}^{(1)}_\bullet, \ldots, \tilde{Y}^{(N)}_\bullet)$ with $\tilde{Y}^{(i)}_\bullet$ centered at $p_i$, 
the divisor $(m + m(D)) \cdot D$ separates $k$-jets at the points $p_1, \ldots, p_N$ simultaneously.
\end{theorem}

\begin{proof}
\textbf{Strategy.} Combining the techniques from the previous two theorems, we first construct an auxiliary ample divisor $m(D) \cdot D - K_X$, then use the subadditivity of Newton-Okounkov bodies and the property that ample divisor bodies contain the origin to translate the assumption on single-point bodies to inclusions for multipoint bodies, and finally apply Nadel vanishing to obtain jet separation.

\textbf{Step 1: Construction of auxiliary ample divisor.}
Since $X$ is a projective variety, $D$ is an ample Cartier divisor, and $K_X$ is a Cartier divisor, by \cite[Example 1.2.10]{lazasfield2004}, there exists $m(D)$ such that $m(D) \cdot D - K_X$ is ample. We decompose
\[
(m + m(D)) \cdot D = K_X + \big(m(D) \cdot D - K_X\big) + mD.
\]

\textbf{Step 2: Subadditivity of Newton-Okounkov bodies.}
For each $j = 1, \ldots, N$ and any infinitesimal flag $\tilde{Y}^{(j)}_\bullet$ at $p_j$, by \cite[Corollary 4.12]{ASENS_2009_4_42_5_783_0}, we have
\[
\Delta_{\tilde{Y}^{(j)}_\bullet}((m(D) \cdot \pi^*(D) - \pi^*(K_X))+m \cdot \pi^*(D)) \supseteq \Delta_{\tilde{Y}^{(j)}_\bullet}(m(D) \cdot \pi^*(D) - \pi^*(K_X))+\Delta_{\tilde{Y}^{(j)}_\bullet}(m \cdot \pi^*(D)).
\]

\textbf{Step 3: Ample divisor body contains the origin.}
Since $m(D) \cdot D - K_X$ is ample, by \cite[Corollary 4.2]{küronya2015infinitesimalnewtonokounkovbodiesjet} and \cite[Definition 2.5]{küronya2015infinitesimalnewtonokounkovbodiesjet}, for any infinitesimal flag $\tilde{Y}^{(j)}_\bullet$ at $p_j$, we have
\[
\mathbf{0} \in \Delta_{\tilde{Y}^{(j)}_\bullet}(m(D) \cdot \pi^*(D) - \pi^*(K_X)).
\]

\textbf{Step 4: From single-point bodies to multi-point bodies.}
The assumption gives $\Delta_{n+k+\varepsilon}^{-1} \subseteq \Delta_{\tilde{Y}_\bullet,j}(m\pi^*(D))$ for all $j = 1, \ldots, N$.

By \cref{lem:InclusionInSinglePointBody}, for any infinitesimal flag $\tilde{Y}^{(j)}_\bullet$ at $p_j$:
\[
\Delta_{\tilde{Y}_\bullet,j}(m\pi^*(D)) \subseteq \Delta_{\tilde{Y}^{(j)}_\bullet}(m\pi^*(D)).
\]

Therefore:
\[
\Delta_{n+k+\varepsilon}^{-1} \subseteq \Delta_{\tilde{Y}^{(j)}_\bullet}(m\pi^*(D)), \quad \forall j = 1, \ldots, N.
\]

\textbf{Step 5: Translation by the origin yields inclusion.}
For any point $y \in \Delta_{n+k+\varepsilon}^{-1}$, by Steps 3 and 4:
\[
y + \mathbf{0} = y \in \Delta_{\tilde{Y}^{(j)}_\bullet}(m(D) \cdot \pi^*(D) - \pi^*(K_X))+\Delta_{\tilde{Y}^{(j)}_\bullet}(m \cdot \pi^*(D)).
\]

Therefore, for all $j = 1, \ldots, N$:
\[
\Delta_{n+k+\varepsilon}^{-1} \subseteq \Delta_{\tilde{Y}^{(j)}_\bullet}(m(D) \cdot \pi^*(D) - \pi^*(K_X))+\Delta_{\tilde{Y}^{(j)}_\bullet}(m \cdot \pi^*(D)).
\]

\textbf{Step 6: Combining with subadditivity.}
By Steps 2 and 5, for all $j = 1, \ldots, N$ and any infinitesimal flag $\tilde{Y}^{(j)}_\bullet$ at $p_j$:
\[
\Delta_{n+k+\varepsilon}^{-1} \subseteq \Delta_{\tilde{Y}^{(j)}_\bullet}((m+m(D)) \cdot \pi^*(D) - \pi^*(K_X)).
\]

This holds for all tuples of infinitesimal flags $\tilde{Y}_\bullet = (\tilde{Y}^{(1)}_\bullet, \ldots, \tilde{Y}^{(N)}_\bullet)$, and therefore also for the multipoint bodies:
\[
\Delta_{n+k+\varepsilon}^{-1} \subseteq \Delta_{\tilde{Y}_\bullet,j}((m+m(D)) \cdot \pi^*(D) - \pi^*(K_X)), \quad \forall j = 1, \ldots, N.
\]

\textbf{Step 7: Application of multipoint jet separation.}
Since $(m+m(D)) \cdot D - K_X$ is ample (hence big), and we have verified that
\[
\Delta_{n+k+\varepsilon}^{-1} \subseteq \Delta_{\tilde{Y}_\bullet,j}((m+m(D)) \cdot \pi^*(D) - \pi^*(K_X))
\]
holds for all $j = 1, \ldots, N$, by \cref{prop:multipoint-NO-jet-separation}, the divisor $K_X + ((m+m(D)) \cdot D - K_X) = (m+m(D)) \cdot D$ separates $k$-jets at the points $p_1, \ldots, p_N$ simultaneously.
\end{proof}

\section{Example: Double Cover of Product of Elliptic Curves}
\label{sec:example}

We use $\ell\geq 2$ for the Koll\'ar--Lazarsfeld parameter in this example. This avoids confusion with the dimension variable in \Cref{thm:main_jet_separation}. In the application below, the variety is the surface $Y$, so $d:=\dim Y=2$; very ampleness corresponds to the jet order $k=1$. We write $s$ for the positive integer denoted by $m$ in the theorem.

\subsection{Setup and ampleness}
Let $S=E\times E$, where $E$ is an elliptic curve. Let $F_1,F_2$ be the two fiber classes and let $\Delta$ be the diagonal. The intersection numbers are
\[
F_1^2=F_2^2=\Delta^2=0,
\qquad
F_1\cdot F_2=F_1\cdot\Delta=F_2\cdot\Delta=1.
\]
Define
\[
A_\ell=\ell F_1+(\ell^2-\ell+1)F_2-(\ell-1)\Delta.
\]
Put $H=F_1+F_2$. Since
\[
A_\ell^2=2,
\qquad
A_\ell\cdot H=\ell^2-2\ell+3>0,
\]
and on an abelian surface the ample cone is the component of the positive cone containing $H$, the divisor $A_\ell$ is ample \cite[Examples~1.5.4 and~1.5.7]{lazasfield2004}.

Set $R=F_1+F_2$, choose a smooth divisor $B\in |2R|$, and let
\[
f:Y\to S
\]
be the double cover branched along $B$. Define $D_\ell=f^*A_\ell$. Since $f$ is finite, $D_\ell$ is ample. Equivalently, by Nakai--Moishezon, if $\Gamma\subset Y$ is irreducible and $C=f(\Gamma)$, then
\[
f_*\Gamma=eC \quad (e=1 \text{ or }2),
\qquad
(f^*L)\cdot\Gamma=L\cdot f_*\Gamma=e(L\cdot C).
\]
Together with $(f^*L)^2=2L^2$, this shows that $f^*L$ is ample on $Y$ if and only if $L$ is ample on $S$.

\subsection{\texorpdfstring{Computation of $m(D_\ell)$}{Computation of m(D ell)}}
The first condition in \Cref{thm:main_jet_separation} is
\[
m(D_\ell)D_\ell-K_Y \quad \text{ample}.
\]
The double-cover formula gives
\[
K_Y=f^*(K_S+R)=f^*R,
\]
since $K_S=0$. For an integer $q$, set $L_q=qA_\ell-R$. Then
\[
qD_\ell-K_Y=f^*L_q,
\]
so $qD_\ell-K_Y$ is ample if and only if $L_q$ is ample on $S$.

Let
\[
N_\ell:=A_\ell\cdot R=\ell^2-2\ell+3.
\]
Using the abelian-surface criterion with the ample class $A_\ell$, $L_q$ is ample if and only if
\[
L_q\cdot A_\ell=2q-N_\ell>0,
\qquad
L_q^2=2q^2-2N_\ell q+2>0.
\]
Equivalently,
\[
q>\frac{N_\ell+\sqrt{N_\ell^2-4}}{2}.
\]
Since $N_\ell\geq 3$, this real number lies strictly between $N_\ell-1$ and $N_\ell$. Hence the minimal integer is
\[
\boxed{m(D_\ell)=N_\ell=\ell^2-2\ell+3.}
\]

\subsection{The Newton--Okounkov condition and the resulting coefficient}
The second condition in \Cref{thm:main_jet_separation} is the inverted-simplex containment
\[
\Delta_{d+k+\varepsilon}^{-1}
\subseteq
\Delta_{\widetilde Y_\bullet}(s\pi_y^*D_\ell)
\]
for every infinitesimal flag over the point. Here $d=2$ and $k=1$, so the required size is $3+\varepsilon$.

The relevant convex-body size is controlled by the infinitesimal Newton--Okounkov body. For any $y\in Y$, put $x=f(y)$. The projection formula above, together with the local multiplicity inequality for finite maps, gives
\[
\varepsilon(D_\ell;y)\geq \varepsilon(A_\ell;x).
\]
On the abelian surface $S$, one has
\[
\varepsilon(A_\ell;x)=1:
\]
the translate of $F_2$ through $x$ gives the upper bound $A_\ell\cdot F_2=1$, while every ample integral line bundle on an abelian variety has Seshadri constant at least $1$ \cite[Example~5.3.10]{lazasfield2004}. By the surface theorem of K\"uronya--Lozovanu \cite{küronya2015infinitesimalnewtonokounkovbodiesjet}, this gives
\[
\Delta_1^{-1}\subseteq \Delta_{\widetilde Y_\bullet}(\pi_y^*D_\ell),
\qquad
\Delta_s^{-1}\subseteq \Delta_{\widetilde Y_\bullet}(s\pi_y^*D_\ell)
\]
for every infinitesimal flag. Therefore, the preceding lower-bound estimate guarantees that the strict condition
\[
\Delta_{3+\varepsilon}^{-1}
\subseteq
\Delta_{\widetilde Y_\bullet}(s\pi_y^*D_\ell)
\]
holds for \(s=4\) and for some \(\varepsilon>0\). Combining this with the value of $m(D_\ell)$, the coefficient in the conclusion of \Cref{thm:main_jet_separation} is
\[
s+m(D_\ell)=4+(\ell^2-2\ell+3)=\ell^2-2\ell+7.
\]
Thus the theorem verifies $1$-jet separation for
\[
\boxed{(\ell^2-2\ell+7)D_\ell.}
\]

\end{document}